\documentclass[a4paper,12pt, reqno]{amsart}

\usepackage{amsmath, amsfonts, amssymb, amsthm, mathrsfs, mathtools}
\usepackage{comment}
\usepackage{enumerate}
\usepackage{enumitem}
\usepackage{fancyhdr}
\usepackage[utf8]{inputenc}
\usepackage[pdfencoding=auto, hyperfootnotes=false, pagebackref]{hyperref}
\renewcommand*{\backref}[1]{}%
\renewcommand*{\backrefalt}[4]{[$\uparrow${\ifcase #1 Not cited.%
          \else \;#2%
          \fi%
    }]}

\usepackage{tikz,tikz-cd}
\usepackage[hang,flushmargin]{footmisc}
\usepackage{graphicx}
\usepackage{xr}
\usepackage[dvistyle, colorinlistoftodos]{todonotes}

\makeatletter
\def\@tocline#1#2#3#4#5#6#7{\relax
  \ifnum #1>\c@tocdepth 
  \else
    \par \addpenalty\@secpenalty\addvspace{#2}%
    \begingroup \hyphenpenalty\@M
    \@ifempty{#4}{%
      \@tempdima\csname r@tocindent\number#1\endcsname\relax
    }{%
      \@tempdima#4\relax
    }%
    \parindent\z@ \leftskip#3\relax \advance\leftskip\@tempdima\relax
    \rightskip\@pnumwidth plus4em \parfillskip-\@pnumwidth
    #5\leavevmode\hskip-\@tempdima
      \ifcase #1
       \or\or \hskip 1em \or \hskip 2em \else \hskip 3em \fi%
      #6\nobreak\relax
    \dotfill\hbox to\@pnumwidth{\@tocpagenum{#7}}\par
    \nobreak
    \endgroup
  \fi}
\makeatother

\newtheorem{theorem}{Theorem}[section]
\newtheorem{lemma}[theorem]{Lemma}

\newtheorem{corollary}[theorem]{Corollary}
\newtheorem{proposition}[theorem]{Proposition}

\newtheorem{conjecture}[theorem]{Conjecture}
\theoremstyle{definition}
\newtheorem{defn}[theorem]{Definition}
\newtheorem{remark}[theorem]{Remark}
\newtheorem{example}[theorem]{Example}

\newcommand{\mb}{\mathbb}

\newcommand{\wh}{\widehat}

\newcommand{\Zmod}[1]{\mb{Z}_{#1}} 

\begin{document}

\title[On sets with small sumset in the circle]{On sets with small sumset in the circle}

\author{Pablo Candela}
\address{Autonomous University of Madrid, and ICMAT\\
	Ciudad Universitaria de Cantoblanco\\
	Madrid 28049\\
	Spain}
\email{pablo.candela@uam.es}
\author{Anne De Roton}
\address{Universit\' e de Lorraine, Institut Elie Cartan de Lorraine, UMR 7502, Vandoeuvre-l\`es-Nancy, F-54506, France.}
\email{anne.de-roton@univ-lorraine.fr}
\date{}
\begin{abstract}
We prove results on the structure of a subset of the circle group having  positive inner Haar measure and doubling constant close to the minimum. 
These results go toward a continuous analogue in the circle of Freiman's $3k-4$ theorem from the integer setting. An analogue of this theorem in $\mb{Z}_p$ has been pursued extensively, and we use some recent results in this direction. For instance, obtaining a continuous analogue of a result of Serra and Z\'emor, we prove that if a subset $A$ of the circle is not too large and has doubling constant at most $2+\varepsilon$ with $\varepsilon<10^{-4}$, then for some integer $n>0$ the dilate $n\cdot A$ is included in an interval in which it has density at least $1/(1+\varepsilon)$. Our arguments yield other variants of this result as well, notably a version for two sets which makes progress toward a conjecture of Bilu. We include two applications of these results. The first is a new upper bound on the size of $k$-sum-free sets in the circle and in $\mb{Z}_p$. The second gives structural information on subsets of $\mb{R}$ of doubling constant at most $3+\varepsilon$.
\end{abstract}
\maketitle \vspace{-0.1cm}
\section{Introduction}
\noindent 
A result of Freiman from 1959 \cite{Freiman1}, often called the $3k-4$ theorem, states that if $A$ is a set of integers such that the sumset $A+A$ satisfies $|A+A|\leq 3|A| - 4$, then $A$ is contained in an arithmetic progression of length $|A+A| - |A| + 1$. This theorem motivated the search for analogues in other settings, especially in groups $\Zmod{p}$ of integers with addition modulo a prime $p$. Treatments of the latter direction include \cite{Freiman2, Gryn, Nat, Rodseth, S-Z}. Part of the difficulty in finding a fully satisfactory $\Zmod{p}$-analogue of the $3k-4$ theorem is that the statement has to involve more assumptions than in the integer setting, in particular to avoid certain counterexamples that occur in $\Zmod{p}$ when $A+A$ is too large. In \cite{S-Z}, Serra and Z\'emor proposed the following conjecture and proved a result towards it (namely \cite[Theorem 3]{S-Z}, which we also recall below).
\begin{conjecture}\label{conj:S-Z}
Let $p$ be a prime, let $r$ be a non-negative integer, and let $A\subset \Zmod{p}$ satisfy
\[
|A+A| = 2|A| + r -1 \leq \tfrac{p}{2}+|A|-2, \;\; \textrm{ and }\;\; r\leq |A|-3.
\]
Then $A$ is included in an arithmetic progression of length $|A| + r$.
\end{conjecture}
\noindent By an \emph{interval} in $\Zmod{p}$ we mean an arithmetic progression of difference 1. For a subset $A$ of an abelian group and an integer $n$, we denote by $n\cdot A$ the image of $A$ under the homomorphism $x\mapsto n\,x$ (for $A\subset \Zmod{p}$ and $n\in \Zmod{p}$ we also use $n\cdot A$ to denote the image of $A$ under $x\mapsto n\,x$). The conclusion of Conjecture \ref{conj:S-Z} can be rephrased as follows: there exists $n\in\Zmod{p}\setminus\{0\}$ and an interval $I\subset \Zmod{p}$ such that $n\cdot A\subset I$ and $|I|\leq |A|+r$.
 
Freiman's $3k-4$ theorem has an extension applicable to two possibly different sets $A$, $B$  \cite{L&S,Stan}. A $\Zmod{p}$-analogue of this extension has also been proposed, namely the so-called \emph{$r$-critical pair conjecture}. A version of this conjecture appeared\footnote{Note that \cite[Conjecture 1]{H-S-Z} appeared before Conjecture \ref{conj:S-Z}, but   in the case $A=B$ it was recognized only later as the likely optimal conjecture, in \cite{S-Z}, thanks to an example given in that paper.} in  \cite{H-S-Z} and was proved for small sets in \cite{B-L-R, Green-Ruzsa}. We recall the following more recent version \cite[Conjecture 19.2]{Gryn}.
\begin{conjecture}\label{conj:asymZp}
Let $p$ be a prime, let $r$ be a non-negative integer, and let $A, B$ be non-empty subsets of $\Zmod{p}$ with $|A|\geq |B|$ and satisfying 
\begin{equation}
|A + B | = |A|+|B|+r-1 \leq \tfrac{1}{2}(p+|A|+|B|)-2, \;\; \textrm{ and }\;\;  r\leq |B|-3.
\end{equation}
Then there exist intervals $I,J,K\subset \Zmod{p}$ and $n\in \Zmod{p}\setminus\{0\}$ such that $n\cdot A\subset I$, $n\cdot B\subset J$,  $n\cdot(A + B) \supset K$, and $|I|\leq |A|+r$, $|J|\leq |B|+r$, $|K|\geq |A| + |B|-1$.
\end{conjecture}
\noindent Note that this extends Conjecture \ref{conj:S-Z} in particular in that the conclusion here concerns not only $A,B$ but also the third set $A+B$.

The following equivalent version of Conjecture \ref{conj:asymZp}, appearing for instance in \cite[Conjecture 19.5]{Gryn}, is notable for its symmetry.

\begin{conjecture}\label{conj:Zptrio}
Let $p$ be a prime, let $r$ be a non-negative integer, and let $A_1,A_2,A_3$ be subsets of $\Zmod{p}$ satisfying the following conditions:
\begin{equation}
|A_1|,|A_2|,|A_3|> r+2, \qquad |A_1|+|A_2|+|A_3|>p-r, \qquad |A_1+A_2+A_3|< p.
\end{equation}
Then there exist intervals $I_1,I_2,I_3\subset \Zmod{p}$ and $n\in \Zmod{p}\setminus\{0\}$ such that $n \cdot A_j \subset I_j$ and $|I_j|\leq |A_j|+r$ for $j=1,2, 3$.
\end{conjecture}
\noindent Considering analogues of the $3k-4$ theorem in the continuous setting of the circle group $\mb{T}=\mb{R}/\mb{Z}$ goes back at least to the paper \cite{FMY} from 1973 by Freiman, Judin, and Moskvin. Conjecture \ref{conj:Zptrio} has a natural analogue in this setting. In this paper, we obtain the following result toward this continuous analogue.
\begin{theorem}\label{thm:Ttrio}
Let $\rho\in (0,c)$ where $c=3.1\cdot 10^{-1549}$. Let $A_1,A_2,A_3$ be subsets of $\mb{T}$ satisfying the following conditions:
\begin{equation}
{\mu}(A_1),{\mu}(A_2),{\mu}(A_3)> \rho, \quad {\mu}(A_1)+{\mu}(A_2)+{\mu}(A_3)>1-\rho, \quad {\mu}(A_1+A_2+A_3)< 1.
\end{equation}
Then there exist closed intervals $I_1,I_2,I_3\subset \mb{T}$ and $n\in \mb{N}$ such that $n \cdot A_j \subset I_j$ and $\mu(I_j)\leq {\mu}(A_j)+\rho$ for $j=1,2,3$.
\end{theorem}
\noindent Here and throughout this paper, we denote by $\mb{N}$ the set of positive integers and by $\mu$ the inner Haar measure on $\mb{T}$, thus for any set $A\subset\mb{T}$ we have that $\mu(A)$ is the supremum of the Haar measures of closed sets included in $A$. We use the inner Haar measure, rather than the Haar measure, in order to deal with non-measurable sets and with the fact that the sumset of two measurable sets can be non-measurable. 

The conjecture mentioned just before Theorem \ref{thm:Ttrio} puts forward that this theorem holds for every $\rho\in (0,1)$.

As we will show, the argument that establishes the equivalence between Conjectures \ref{conj:asymZp} and \ref{conj:Zptrio} can be adapted to the continuous setting, by incorporating several additional technicalities, to show that Theorem \ref{thm:Ttrio} implies the following result. 
\begin{theorem}\label{thm:Tasym}
Let $\rho\in (0,c)$ where $c=3.1\cdot 10^{-1549}$. Let $A,B\subset \mb{T}$ satisfy
\[
{\mu}(A+B) \; = \; {\mu}(A)+{\mu}(B)+ \rho \, < \; \tfrac{1}{2}\big(1+{\mu}(A)+{\mu}(B)\big), \quad \textrm{ and }\quad \rho < {\mu}(B)\leq {\mu}(A).
\]
Then there exist intervals $I,J,K\subset \mb{T}$, with $I,J$ closed, $K$ open, $n\in \mb{N}$, such that $n\cdot A \subset I$, $n\cdot B \subset J$, $K\subset n\cdot (A+B)$, and $\mu(I)\leq {\mu}(A)+\rho$, $\mu(J)\leq {\mu}(B)+\rho$, $\mu(K)\geq {\mu}(A)+{\mu}(B)$.
\end{theorem}
\noindent This theorem makes progress toward an analogue of Conjecture \ref{conj:asymZp} for $\mb{T}$, analogue which originated in work of Bilu on the so-called \emph{$\alpha+2\beta$ inequality} in the torus (see \cite[Conjecture 1.2]{Bilu}). We detail this in Remark \ref{rem:BiluConj} in Section \ref{sec:mainpfs}, after having proved Theorem \ref{thm:Tasym}.

The bound $c=3.1\cdot 10^{-1549}$ in Theorems \ref{thm:Ttrio} and \ref{thm:Tasym} comes from a result in $\mb{Z}_p$ due to Grynkiewicz, as we explain in Section \ref{sec:mainpfs}. In the \emph{symmetric case}, i.e. when $A=B$, a better bound had been given by the result of Serra and Z\'emor toward Conjecture \ref{conj:S-Z} (result recalled as Theorem \ref{thm:Zpsym} below). We prove the following $\mb{T}$-analogue of this result.
\begin{theorem}\label{thm:Tsym}
Let $0\leq \varepsilon\leq 10^{-4}$. Let $A\subset \mb{T}$ satisfy ${\mu}(A)>0$ and 
\[
{\mu}(A+A)= (2+\varepsilon)\, {\mu}(A)\; <\;\tfrac{1}{2} +{\mu}(A).
\]
Then there exist intervals $I,K\subset \mb{T}$, with $I$ closed, $K$ open and $n\in \mb{N}$, such that $n\cdot A \subset I$, $K\subset n\cdot (A+A)$, $\mu(I)\leq {\mu}(A+A)-{\mu}(A)$ and $\mu(K)\geq 2{\mu}(A)$.
\end{theorem}
\noindent Apart from their relation to continuous analogues of the $3k-4$ theorem and Bilu's conjecture, the theorems above are motivated by the following applications.

The first application concerns the problem of determining the supremum of measures of Borel sets $A\subset \mb{T}$ such that the cartesian power $A^3$ contains no triple $(x,y,z)$ solving the equation $x+y=kz$, where $k\geq 3$ is a fixed integer. This is an analogue in $\mb{T}$ of a problem which goes back to Erd\H os (see \cite{C&G}) and which has been treated in several works, first in the integer setting (see in particular \cite{B&al, C&G}) and then also in the continuous setting of an interval in $\mb{R}$  \cite{C&G-cts, M&R, D&P}. The above-mentioned supremum is seen to be at most $1/3$ by a simple application of Raikov's inequality from \cite{Raikov} (see also \cite[Theorem 1]{Macbeath}). Our result, discussed in Section \ref{sec:kfs}, improves on this upper bound using Theorem \ref{thm:Tsym}; see Theorem \ref{thm:ksfs-ub}. Via a correspondence established in \cite{CS} between this problem in $\mb{T}$ and a similar problem in $\Zmod{p}$, Theorem \ref{thm:ksfs-ub} implies a similar result in $\Zmod{p}$; see Remark \ref{rem:TtoZp}.

The second application provides new results about the structure of subsets of $\mb{R}$ of doubling less than 4. We discuss this in Section \ref{sec:dlt4}. Essentially, if a closed set $A\subset [0,1]$ has doubling constant at most $3+\varepsilon$, then modulo 1 it has doubling constant at most $2+\varepsilon$, and so Theorem \ref{thm:Tsym} can be used to obtain information on the structure of $A$; see Theorem \ref{thm:small_doubling_R}. In particular, under a special case of the conjecture of Bilu mentioned above \cite[Conjecture 1.2]{Bilu}, we obtain a version of \cite[Theorem 6.2]{EGM} with effective bounds; see Corollary \ref{cor:EGM}. 

\vspace{0.6cm}

\noindent \textbf{Acknowledgements.} The authors are very grateful to Imre Ruzsa, Oriol Serra and Gilles Z\'emor for useful comments. This work was  supported by project ANR-12-BS01-0011 CAESAR and by grant MTM2014-56350-P of MINECO.

\section{Proofs of the main results}\label{sec:mainpfs}

\noindent As mentioned in the introduction, the analogues in $\Zmod{p}$ of Theorems \ref{thm:Ttrio} and \ref{thm:Tsym} are known.

Indeed, Theorem \ref{thm:Ttrio} is a $\mb{T}$-analogue of the following result.

\begin{theorem}\label{thm:Zptrio}
Let $p$ be a prime, and let $r$ be an integer with $0\leq r \leq c\, p-1.2$ where $c=3.1\cdot 10^{-1549}$. Let $A_1,A_2,A_3$ be subsets of $\Zmod{p}$ satisfying the following conditions:
\begin{equation}
|A_1|,\, |A_2|, \,|A_3|> r+2, \quad |A_1|+|A_2|+|A_3|> p-r, \quad |A_1+A_2+A_3|< p.
\end{equation}
Then there exist intervals $I_1,I_2,I_3\subset \Zmod{p}$ and a non-zero $n\in \Zmod{p}$ such that $n \cdot A_j \subset I_j$ and $|I_j|\leq |A_j|+r$, for $j=1,2,3$.
\end{theorem}
\noindent For a subset $A$ of an abelian group $G$ we denote by $A^c$ the complement $G\setminus A$. Theorem \ref{thm:Zptrio} can be deduced from the following result of Grynkiewicz (see \cite[Theorem 21.8]{Gryn}).
\begin{theorem}\label{thm:ZpGryn}
Let $p$ be a prime, and let $r$ be an integer with $0\leq r \leq c\,p-1.2$ where $c=3.1\cdot 10^{-1549}$. Let $A,B$ be subsets of $\Zmod{p}$ satisfying the following conditions:
\begin{equation}
|A|,\, |B|,\, |C|> r+2, \qquad |A+B|\leq |A|+|B|+r-1,
\end{equation}
where $C=- (A+B)^c$. Then there exist intervals $I,J,K\subset \Zmod{p}$ and a non-zero $n\in \Zmod{p}$ such that $n \cdot A \subset I$, $n \cdot B \subset J$, $n \cdot C \subset K$, and $|I|\leq |A|+r$, $|J|\leq |B|+r$, $|K|\leq |C|+r$.
\end{theorem}

\begin{lemma}
Theorem \ref{thm:ZpGryn} implies Theorem \ref{thm:Zptrio}.
\end{lemma}
\begin{proof}
Starting from the assumptions in Theorem \ref{thm:Zptrio}, note that since $|A_1+A_2+A_3|<p$ we may assume (modulo translating $A_3$, which does not affect the theorem) that $0\not\in A_1+A_2+A_3$. Hence $A_3\subset (-A_1-A_2)^c= -(A_1+A_2)^c$. Let $A=A_1$, $B=A_2$, $C=-(A_1+A_2)^c$, and note that $|A|,\,|B|,\,|C|> r+2$. Moreover, from $|A_1|+|A_2|+|A_3|>p-r$ we deduce that $|A+B|\leq |A|+|B|+r-1$. Let $s\leq r$ be such that $|A+B|=|A|+|B|+s-1$. Applying Theorem \ref{thm:ZpGryn} with $s$, we obtain intervals $I_1=I$, $I_2=J$, $I_3=K$ and $n\in\Zmod{p}\setminus\{0\}$ such that $n\cdot A_j\subset I_j$ and $|I_j|\leq |A_j|+s\leq |A_j|+r$ for $j=1,2$. Moreover $n\cdot A_3\subset n\cdot C\subset I_3$, and $|I_3|\leq |C|+s = p-|A_1+A_2|+s = p-|A_1|-|A_2|+1 \leq |A_3|+r$, so we obtain the conclusion of Theorem \ref{thm:Zptrio}.
\end{proof}
\noindent One can also deduce Theorem \ref{thm:ZpGryn} from Theorem \ref{thm:Zptrio} in a straightforward way; we leave this to the reader, and in any case the main ideas in this deduction will be used in the continuous setting in Subsection \ref{subsec:gen}, to prove Theorem \ref{thm:Tasym}.

In the case $A=B$ of Theorem \ref{thm:ZpGryn} (the symmetric case), the following result of Serra and Z\'emor toward their Conjecture \ref{conj:S-Z} provided a better bound for $r$ than in Theorem \ref{thm:ZpGryn} (see \cite[Theorem 3]{S-Z}).

\begin{theorem}\label{thm:Zpsym}
Let $p$ be a prime greater than $2^{94}$, let $0\leq \varepsilon \leq 10^{-4}$, and let $A\subset \Zmod{p}$ satisfy
\begin{equation}
|A+A| = (2+\varepsilon)|A|-1 \leq  \min\big\{\,3|A|-4, \; \tfrac{p}{2}+|A|-2\,\big\}.
\end{equation}
Then there is an interval $I\subset \Zmod{p}$ and $n \in \Zmod{p}\setminus\{0\}$ such that $n\cdot A \subset I$ and $|I|\leq |A+A|-|A|+1$.
\end{theorem}

\noindent In this section we prove Theorems \ref{thm:Ttrio}, \ref{thm:Tasym} and \ref{thm:Tsym}. Inspired by arguments of Bilu from \cite{Bilu}, we deduce the first and the third of these theorems from their discrete versions, i.e. Theorems  \ref{thm:Zptrio} and \ref{thm:Zpsym} respectively. In the process, we also deduce Theorem \ref{thm:Tasym} using Theorem \ref{thm:Ttrio}.

Let $\tfrac{1}{p}\Zmod{p}$ denote the subgroup of $\mb{T}$ isomorphic to $\Zmod{p}$. We use the following notation for discrete approximations of sets in $\mb{T}$. 
\begin{defn}
For any set $A\subset \mb{T}$ and prime $p$, we define the set
\[
A_p= A\, \cap \, \tfrac{1}{p}\Zmod{p}.
\]
\end{defn}
\noindent To prove Theorem \ref{thm:Ttrio}, first we focus on sets in $\mb{T}$ that are unions of finitely many intervals. We refer to such sets as \emph{simple sets}. In Subsection \ref{subsec:sos} we show that if $A_1,A_2,A_3$ are open simple sets and satisfy the conditions of Theorem \ref{thm:Ttrio}, then their discrete approximations $A_{j,p}$ obey the conclusion of Theorem \ref{thm:Zptrio}. However, this is not enough to deduce directly that the conclusion of Theorem \ref{thm:Ttrio} holds for the original sets $A_j$, because the integer $n$ provided by Theorem  \ref{thm:Zptrio} is not a priori bounded in any way that would ensure  that the sets $n\cdot A_j$ are contained in suitably small intervals the way their discrete approximations are. To ensure this additional fact, in the next subsection we use the Fourier transform on $\Zmod{p}$ to bound the integer $n$. Finally, in Subsection \ref{subsec:gen} we obtain Theorems \ref{thm:Ttrio} and also \ref{thm:Tasym} and \ref{thm:Tsym}, by generalizing from simple sets to arbitrary sets.

\medskip

\subsection{On the $n$-diameter of simple sets}\hfill \medskip\\
Given a set $A\subset \mb{T}$ and an integer $n$, we define the \emph{$n$-diameter} of $A$ by
\begin{equation}
D_n(A)=\inf \{\mu(I): I\subset \mb{T}\textrm{ a closed interval such that }n\cdot A\subset I\}.
\end{equation}
For a set $B\subset \Zmod{p}$ and $n\in \mb{Z}_p$, we define similarly the \emph{$n$-diameter} of $B$ by
\[
D_n(B)=\min \{|I|/p: I\subset \Zmod{p}\textrm{ an interval such that }n\cdot B\subset I\}.
\]
\noindent We prove the following result concerning the $n$-diameter of simple sets in $\mb{T}$.
\begin{proposition}\label{prop:SimpSetDiamT}
Let $A\subset \mb{T}$ be a union of at most $m$ intervals with $\mu(A)>0$, and suppose that $n\in \mb{Z}$ satisfies $D_n(A)<\min\big(\tfrac{1}{2},\tfrac{\mu(A)}{1-2/\pi}\big)$. Then $|n|\leq \tfrac{m}{2\big(\mu(A) -(1-2/\pi)D_n(A)\big)}$.
\end{proposition}

\noindent For $s\in \Zmod{p}$ we denote by $|s|_p$ the absolute value of the unique integer in $(-\tfrac{p}{2},\tfrac{p}{2})$ congruent to $s$ modulo $p$. We deduce Proposition \ref{prop:SimpSetDiamT} from the following discrete version, which is in fact the main result from this subsection that we use in the sequel.

\begin{proposition}\label{prop:SimpSetDiamZp}
Let $B\subset \Zmod{p}$ be a union of at most $m$ intervals with $\tfrac{|B|}{p}=\beta >0$, and suppose that $n\in \Zmod{p}$ satisfies $D_n(B)<\min\big(\tfrac{1}{2},\tfrac{\beta}{1-2/\pi}\big)$. Then $|n|_p\leq \tfrac{m}{2\big(\beta -(1-2/\pi)D_n(B)\big)}$.
\end{proposition} 
\noindent Proposition \ref{prop:SimpSetDiamT} follows by applying Proposition \ref{prop:SimpSetDiamZp} to $B=A_p$ for primes $p\to \infty$.

To obtain Proposition \ref{prop:SimpSetDiamZp}, we use the following result concerning the Fourier coefficients of a subset of $\Zmod{p}$ that is the union of at most $m$ disjoint intervals, to the effect that these Fourier coefficients decay in a useful way.

For $f:\Zmod{p}\rightarrow \mb{C}$, let $\wh{f}$ denote the Fourier transform $\wh{\Zmod{p}}\cong \Zmod{p}\to \mathbb{C}$ defined by $\wh{f}(s)=\frac1p\sum_{j\in\Zmod{p}}f(s)e^{2\pi i\, \frac{sj}{p}}$. We write $[m]$ for the set of integers $\{1,\ldots,m\}$. 

\begin{lemma}
Let $J_1,J_2,\dots,J_m$ be pairwise disjoint intervals in $\Zmod{p}$, and let $B = \bigsqcup_{i\in [m]} J_i$. Let $s$ be a non-zero element of $\wh{\Zmod{p}}\cong \Zmod{p}$. Then we have
\begin{equation}\label{eq:FourierJordan}
|\wh{1_B}(s)| \leq \frac{m}{2 |s|_p}.
\end{equation}
\end{lemma}

\begin{proof}
We first estimate the Fourier coefficients of a single interval $J\subset \Zmod{p}$, by the following standard calculation. Supposing that $J=\{a,a+1,\dots,a+(t-1)\}$, for every non-zero $s\in (-\frac{p}{2},\frac{p}{2})$ we have
\begin{equation}\label{eq:geomcalc1}
|\wh{1_J}(s)| = \Big| \tfrac{1}{p} \sum_{j=0}^{t-1} e^{2\pi i \frac{s(a+j)}{p}}\Big| = \tfrac{1}{p} \Big|  \sum_{j=0}^{t-1} e^{2\pi i\, \frac{sj}{p}}\Big| = \tfrac{1}{p} \frac{|1-e^{2\pi i\, \frac{st}{p}}|}{|1-e^{2\pi i\, \frac{s}{p}}|} \leq \tfrac{1}{p} \frac{2}{|1-e^{2\pi i\, \frac{s}{p}}|}.
\end{equation}
Letting $\|\theta\|_{\mb{T}}$ denote the distance from $\theta\in \mb{R}$ to the nearest integer, and using the standard estimate $|1-e^{2\pi i\,\theta}| \geq 4 \|\theta\|_{\mb{T}}$ for $\|\theta\|_{\mb{T}}<1/2$, we deduce that
\begin{equation}\label{eq:FourierInterval}
|\wh{1_J}(s)| \leq \frac{1}{2\,|s|_p}.
\end{equation}
Now, since $1_B = 1_{J_1}+\cdots+1_{J_m}$, we deduce \eqref{eq:FourierJordan} by linearity of the Fourier transform, the triangle inequality, and applying \eqref{eq:FourierInterval} to each interval $J_i$.
\end{proof}
\noindent An immediate consequence of this lemma is that for such a set $B$ the large Fourier coefficients can only occur at bounded frequencies, in the following sense.

\begin{corollary}\label{cor:FreqUB}
Let $J_1,\dots,J_m \subset \Zmod{p}$ be pairwise disjoint intervals, let $B = \bigsqcup_{i\in [m]} J_i$, and let $\gamma>0$. If $s\in \Zmod{p}$ satisfies $|\wh{1_B}(s)| \geq \gamma$, then $|s|_p\leq \frac{m}{2\gamma}$.
\end{corollary}
\begin{proof}
We may assume that $s\neq 0$, and then by \eqref{eq:FourierJordan} we have $ \frac{m}{2|s|_p} \geq |\wh{1_B}(s)| \geq \gamma$, whence the result follows.
\end{proof}
We shall combine this corollary with the following result.
\begin{lemma}\label{lem:FourierLB}
Let $B\subset \Zmod{p}$, let $n\in \Zmod{p}\setminus\{0\}$, and let $I$ be an interval in $\Zmod{p}$ such that $n\cdot B\subset I$ and $|I|< p/2$. Then
\begin{equation}\label{eq:FourierLB}
\big|\wh{1_B}(n)\big| > \tfrac{1}{p}\big(|B| - (1-\tfrac{2}{\pi})|I|\big).
\end{equation}
\end{lemma}
\noindent This lemma yields a positive lower bound for $\big|\wh{1_B}(n)\big|$ when $|B|/|I| > 1-\tfrac{2}{\pi} \lesssim 0.364$.
\begin{proof}
We have
\begin{eqnarray*}
\wh{1_B}(n) & = & \tfrac{1}{p}\sum_{j\in \Zmod{p}} 1_B(j)\, e^{2\pi i\, \frac{nj}{p}} = \tfrac{1}{p}\sum_j 1_B(n^{-1} j) \, e^{2\pi i\, \frac{j}{p}} =  \tfrac{1}{p}\sum_j 1_{n\cdot B}(j)\,  e^{2\pi i\, \frac{j}{p}} \\
& = &  \tfrac{1}{p}\sum_j 1_I(j) \, e^{2\pi i\, \frac{j}{p}} + \tfrac{1}{p}\sum_j (1_{n\cdot B}(j)-1_I(j)) \, e^{2\pi i\, \frac{j}{p}}.
\end{eqnarray*}
The last sum here has magnitude at most $\tfrac{1}{p} |I\setminus n\cdot B|$. We may assume that $I=\{0,1,\dots,(t-1)\}$ for some $t< p/2$. Hence
\begin{eqnarray*}
| \wh{1_B}(n) | & \geq & \tfrac{1}{p} \Big(\Big|\sum_{j \in I} e^{2\pi i\, \frac{j}{p}} \Big|-|I\setminus n\cdot B|\Big) =  \tfrac{1}{p}\Big( \frac{|1-e^{2\pi i\, \frac{t}{p}}|}{|1-e^{2\pi i\, \frac{1}{p}}|}+|B|-|I|\Big)
\end{eqnarray*}
where we have used the same calculation as in \eqref{eq:geomcalc1}. Using the estimates
\[
4 \|\theta\|_{\mb{T}} \leq |1-e^{2\pi i\,\theta}| \leq 2\pi \|\theta\|_{\mb{T}}\;\textrm{ for }\|\theta\|_{\mb{T}}<1/2,
\]
we obtain $\frac{|1-e^{2\pi i\, \frac{t}{p}}|}{|1-e^{2\pi i\, \frac{1}{p}}|} \geq \frac{4 t/p}{2\pi/p}=\frac{2}{\pi}|I|$, and the result follows.
\end{proof}

\begin{proof}[Proof of Proposition \ref{prop:SimpSetDiamZp}]
By definition of $D_n(B)$ there is an interval $I\subset \Zmod{p}$ satisfying $\frac{|I|}{p}=D_n(B)<\tfrac{1}{2}$ and $n\cdot B\subset I$. Lemma \ref{lem:FourierLB} gives us $|\wh{1_B}(n)|> \tfrac{1}{p}\big(|B|- (1-\tfrac{2}{\pi})|I|\big)$, and this lower bound is positive by our assumptions. Combining this with Corollary \ref{cor:FreqUB}, we obtain $|n|_p\leq \frac{m}{2\big(\beta- (1-{2}/{\pi})D_n(B)\big) }$, as claimed.
\end{proof}

\begin{remark}
Some restriction on the size of $D_n(B)$ is necessary in Lemma \ref{lem:FourierLB} and in Proposition \ref{prop:SimpSetDiamZp}. Indeed, if $B=I=\{0,\ldots,t-1\}$ with $t=p(1-\theta)$ and $0<\theta< 1/2$, then $D_1(B)=|B|/p=1-\theta$. In this case, with $n=1$ we have
\[
|\wh{1_B}(n)| =|\wh{1_I}(1) |=\tfrac{1}{p} \frac{|1-e^{2\pi i\, \frac{t}{p}}|}{|1-e^{2\pi i\, \frac{1}{p}}|}=\tfrac{1}{p} \frac{\big|\sin\big(\pi \frac{t}{p}\big)\big|}{\sin\big(\frac{\pi}{p}\big)}=\tfrac{1}{p} \frac{\sin\left(\pi \theta\right)}{\sin\big(\frac{\pi}{p}\big)}.
\]
For $p$ large, this is very close to $\tfrac1\pi\sin(\pi\theta)< 1/\pi$, whereas $\tfrac{1}{p}\big(|B| - (1-\tfrac{2}{\pi})|I|\big)=\tfrac{2}{\pi}(1-\theta)> 1/\pi$. This shows that  \eqref{eq:FourierLB} can fail if $D_n(B)>1/2$.

Proposition \ref{prop:SimpSetDiamZp} fails for this set $B$ also if $\beta=1-\theta>\pi/4$, since in this case we have $D_1(B)=\beta$, $m=1$, and yet $1>\tfrac{\pi}{4\beta}=\tfrac{1}{2\big(\beta -(1-2/\pi)D_1(B)\big)}$.
\end{remark}

\medskip

\subsection{Proof of the main result for simple open sets}\label{subsec:sos}\hfill \medskip\\
In this subsection we establish Theorem \ref{thm:Ttrio} for simple open sets as follows.
\begin{proposition}\label{prop:SimpSet}
Let $\rho\in (0,c)$ where $c=3.1\cdot 10^{-1549}$. For each $j\in [3]$ let $A_j\subset \mb{T}$ be a union of at most $m$ pairwise disjoint open intervals, and suppose that
\begin{equation}\label{eq:SimpSet}
\min_j\mu(A_j)>\rho, \qquad \mu(A_1)+\mu(A_2)+\mu(A_3)>1-\rho, \qquad \mu(A_1+A_2+A_3)<1.
\end{equation}
Then there exists a positive integer $n\leq \frac{2m}{\min_j \mu(A_j)}$ and closed intervals $I_1,I_2,I_3\subset \mb{T}$ such that $n\cdot A_j\subset I_j$ and $\mu(I_j)\leq \mu(A_j)+\rho$ for $j\in [3]$.
\end{proposition}

\begin{proof}
\noindent Fix any $\delta>0$ satisfying
\[
\delta< \min \big\{\tfrac{1}{2}\big(\mu(A_1)+\mu(A_2)+\mu(A_3)-(1-\rho)\big), \; 1-\mu(A_1+A_2+A_3), \; \rho/10\big\}.
\]
For $p$ sufficiently large, we can assume that $A_{j,p}$ is the union of at most $m$ intervals in $\tfrac{1}{p}\Zmod{p}$. Let $\mu_p$ denote the discrete measure $\frac{1}{p}\sum_{j=0}^{p-1} \delta_{j/p}$ on $\mb{T}$, where $\delta_{j/p}$ is a Dirac $\delta$ measure at $j/p$. We have that $\mu_p$ converges weakly to the Haar probability measure on $\mb{T}$ as $p\to \infty$, and so $\tfrac{1}{p}|A_{j,p}|=\mu_p(A_{j,p})\to \mu(A_j)$ and $\mu_p((A_1+A_2+A_3)_p)\to \mu(A_1+A_2+A_3)$ as $p\to \infty$. Note that the inner Haar measure and the Haar measure of simple sets in $\mb{T}$ coincide and that a sumset of simple sets is a simple set. In particular, for $p$ sufficiently large (depending on the sets $A_j$ and $\delta$) we have
\begin{equation}\label{eq:estmumupS}
|\mu_p\big(A_{j,p})-\mu(A_j)|\leq \delta/3, \qquad |\mu_p\big((A_1+A_2+A_3)_p \big)-\mu(A_1+A_2+A_3)| \leq \delta.
\end{equation}
By our assumptions in \eqref{eq:SimpSet}, the fact that $A_{1,p}+A_{2,p}+A_{3,p}\subset (A_1+A_2+A_3)_p$, and our choice of $\delta$ and $p$, we then have
\begin{equation}\label{eq:SimpSetZp}
\mu_p\big(A_{j,p}\big)>\rho-\delta/3, \quad \sum_{j\in [3]}\mu_p\big(A_{j,p}\big)>1-\rho+\delta, \quad \mu_p\big(A_{1,p}+A_{2,p}+A_{3,p}\big)<1.
\end{equation}
Let $r$ be an integer in $((\rho-\delta)p,(\rho-2\delta/3)p\big)$. For $p\geq 6/\delta$ sufficiently large we can apply Theorem \ref{thm:Zptrio} to the sets $A_{j,p}$ with this integer $r$. This yields intervals $I_{j,p}\subset \tfrac{1}{p}\Zmod{p}$ and a non-zero integer $n\in (-\frac{p}{2},\frac{p}{2})$ such that $n\cdot A_{j,p}\subset I_{j,p}$ and $\mu_p(I_{j,p})\leq \mu_p(A_{j,p})+r/p\leq \mu(A_j)+\rho-\delta/3$. For every $x$ in the simple open set $A_j$, there exists $y\in A_{j,p}$ such that $\|x-y\|_{\mb{T}}\leq \tfrac{1}{2p}$, which implies that $\|nx-ny\|_{\mb{T}}\leq \tfrac{|n|}{2p}$. Hence
\[
n \cdot A_j\,\subset\, n\cdot A_{j,p} + \big[-\tfrac{|n|}{2p},\tfrac{|n|}{2p}\big] \subset I_{j,p}+\big[-\tfrac{|n|}{2p},\tfrac{|n|}{2p}\big].
\]
Let $I_j$ be the closed interval $I_{j,p}+\big[-\tfrac{|n|}{2p},\tfrac{|n|}{2p}\big]$. We have
\begin{equation}\label{eq:IjUB}
\mu(I_j)\leq \mu_p(I_{j,p}) +\tfrac{|n|}{p} \leq \mu(A_j)+\rho-\delta/3+\tfrac{|n|}{p}.
\end{equation}
We have $\min_j\mu(A_j)<\tfrac{1}{3}$, by the third inequality in \eqref{eq:SimpSet} and Raikov's inequality \cite[Theorem 1]{Macbeath}. Therefore, supposing without loss of generality that $\min_j\mu(A_j) = \mu(A_1)$, we have $\mu_p(I_{1,p})<1/2$. We also have $\mu_p(I_{1,p})< \tfrac{\mu_p(A_{1,p})}{1-2/\pi}$. By Proposition \ref{prop:SimpSetDiamZp}, we conclude that $|n|\leq \tfrac{m}{2\big(\mu_p(A_{1,p})-(1-2/\pi)\mu_p(I_{1,p})\big)}$. Since $\mu_p(I_{1,p})< 2\mu(A_1)-\delta/3$ and $\delta<\rho/10$, we have 
\[
|n|\leq \frac{m}{2\big(\mu(A_1)-\delta/3-(1-2/\pi)(2\mu(A_1)-\delta/3)\big)} \leq  \frac{m}{2\big((\tfrac{4}{\pi}-1)\mu(A_1)-\tfrac{2\delta}{3\pi}\big)}\leq  \frac{2m}{\mu(A_1)}.
\]
Since we can take $p>\tfrac{6m}{\delta\,\rho}$, we have $\tfrac{2m}{\mu(A_1)}\leq \tfrac{\delta\, p}{3}$, and so from \eqref{eq:IjUB} we have $\mu(I_j)\leq \mu(A_j)+\rho$. Finally, as $n\cdot A_j$ and $-n\cdot A_j$ are both included in suitable intervals, we can have $n>0$.
\end{proof}

\medskip

\subsection{From simple sets to arbitrary sets}\label{subsec:gen}\hfill \medskip \\
In this subsection we deduce Theorem \ref{thm:Ttrio} using Proposition \ref{prop:SimpSet}. To that end, we first prove Theorem \ref{thm:Ttrio} for closed sets.

\begin{proposition}\label{prop:Ttrioclosed}
Let $\rho\in (0,c)$ where $c=3.1\cdot 10^{-1549}$. Let $A_1,A_2,A_3$ be closed subsets of $\mb{T}$ satisfying the following conditions:
\begin{equation}
\mu(A_1),\mu(A_2),\mu(A_3)> \rho, \quad \sum_{j\in [3]}\mu(A_j)>1-\rho, \quad \mu(A_1+A_2+A_3)< 1.
\end{equation}
Then there exist closed intervals $I_1,I_2,I_3\subset \mb{T}$ and a positive integer $n$ such that $n \cdot A_j \subset I_j$ and $\mu(I_j)\leq \mu(A_j)+\rho$ for $j\in [3]$.
\end{proposition}

\begin{proof}
Fix any $\varepsilon>0$ with $\varepsilon< \min \big\{\mu(A_1)+\mu(A_2)+\mu(A_3)-(1-\rho), \; 1-\mu(A_1+A_2+A_3)\big\}$. For $\delta>0$ let $I_\delta$ denote the open interval $(-\delta,\delta)$ in $\mb{T}$. We have $A=\cap_{\delta>0}(A+I_\delta)$ and $A_1+A_2+A_3=\cap_{\delta>0}(A_1+A_2+A_3+I_{3\delta})=\cap_{\delta>0}(A_1+I_\delta+A_2+I_\delta+A_3+I_\delta)$. Let $\delta>0$ be sufficiently small so that
\[
\forall\,j\in [3],\;\mu(A_j+I_\delta)\leq \mu(A_j)+\varepsilon, \quad\textrm{and}\quad \mu(A_1+A_2+A_3+I_{3\delta})\leq \mu(A_1+A_2+A_3)+\varepsilon <1.
\]
By compactness of each set $A_j$, there exists a set $A_j'$ that is the union of finitely many translates of $I_\delta$ such that $A_j\subset A'_j\subset A_j+I_\delta$. The simple open sets $A'_1,A'_2,A'_3$ satisfy the inequalities in \eqref{eq:SimpSet} with initial parameter $\rho-\varepsilon$. Therefore, by Proposition \ref{prop:SimpSet} applied to these sets with this parameter, we obtain a positive integer $n$ and closed intervals $I_1,I_2,I_3 \subset\mb{T}$ with $n\cdot A_j\subset n\cdot A'_j\subset I_j$, and $\mu(I_j)\leq \mu(A'_j)+\rho-\varepsilon \leq \mu(A_j)+\rho$.
\end{proof}

\noindent From this we shall deduce Theorems \ref{thm:Ttrio} and \ref{thm:Tasym}. To do so we use the following lemma based on ideas of Bilu from \cite{Bilu}.

\begin{lemma}\label{lem:clapprox}
Let $C\subset A\subset \mb{T}$, where $C$ is a closed set, and let $X$ be a finite set of integers. Then for every $\varepsilon>0$ there exists a closed set $E$ with $C\subset E\subset A$ such that $\mu(E)=\mu(C)$ and $n\cdot A\subset n\cdot E+[-\varepsilon,\varepsilon]$ for all $n\in X$. 
\end{lemma}

\begin{proof}
For each $n\in X$, since $A$ is totally bounded, there is a finite subset $E(n,\varepsilon)\subset A$ such that $A\subset E(n,\varepsilon)+[-\tfrac{\varepsilon}{n},\tfrac{\varepsilon}{n}]$, and so $n\cdot A \subset n\cdot E(n,\varepsilon)+[-\varepsilon,\varepsilon]$. The set $E = C \cup \big(\bigcup_{n\in X} E(n,\varepsilon)\big)$ satisfies the claim in the lemma.
\end{proof}

\noindent We shall combine this with the following special case of \cite[Lemma 4.2.1]{Bilu}.

\begin{lemma}\label{lem:finset}
Let $B$ be a Haar measurable subset of $\mb{T}$ with $\mu(B)>0$, and let $\lambda < 1$. Then there exist only finitely many integers $n$ such that $\mu(n\cdot B)\leq \lambda$.
\end{lemma}

We can now prove the main result.

\begin{proof}[Proof of Theorem \ref{thm:Ttrio}]
We assume that ${\mu}(A_j)>\rho$, that ${\mu}(A_1)+{\mu}(A_2)+{\mu}(A_3)>1-\rho$, and that ${\mu}(A_1+A_2+A_3)<1$. Fix an arbitrary $\delta$ satisfying 
\[
0< \delta <  \min \big\{\; \tfrac{1}{2} \big( {\mu}(A_1)+{\mu}(A_2)+{\mu}(A_3)-(1-\rho) \big),\;\rho\; \big\}.
\]
We may assume that ${\mu}(A_1)=\min_{j\in [3]} {\mu}(A_j)$. Let $C_1$ be a closed subset of $A_1$ such that $\mu(C_1)> {\mu}(A_1)-\delta/3 >0$. Let $X$ be the set of integers $n$ such that $\mu(n\cdot C_1)\leq {\mu}(A_1)+\rho$. Since ${\mu}(A_1+A_2+A_3)<1$, by Raikov's inequality we have ${\mu}(A_1)< 1/3$, and so we can certainly apply Lemma \ref{lem:finset} to deduce that $X$ is finite. Let $A_1'$ be the closed subset of $A_1$ obtained by applying Lemma \ref{lem:clapprox} to $C_1\subset A_1$ with $\varepsilon=\delta/2$. Similarly, for $j=2,3$ let $A'_j$ be a closed subset of $A_j$ such that $\mu(A'_j)\geq {\mu}(A_j)-\delta/3$ and $n\cdot A_j\subset n\cdot A_j'+[-\tfrac{\delta}{2},\tfrac{\delta}{2}]$ for all $n\in X$. We then have $\mu(A'_j)>\rho-\delta/3$ for every $j$, we have $\mu(A_1'+A_2'+A_3')<1$, and
\[
\mu(A_1')+\mu(A_2')+\mu(A_3')>{\mu}(A_1)+{\mu}(A_2)+{\mu}(A_3)-\delta >1-(\rho-\delta).
\]
By Proposition \ref{prop:Ttrioclosed} applied to the sets $A_j'$ with initial parameter $\rho-\delta$, there exist closed intervals $I_j'$ such that $\mu(I_j')\leq\mu(A'_j)+\rho-\delta \leq  {\mu}(A_j)+\rho-\delta$, and a positive integer $n$ such that $n\cdot A'_j\subset I_j'$, for $j=1,2,3$. In particular we must have $n\in X$, and then by our choice of the sets $A_j'$ we have $n\cdot A_j\subset I_j'+[-\tfrac{\delta}{2},\tfrac{\delta}{2}]$. Letting $I_j$ be the closed interval $I_j'+[-\tfrac{\delta}{2},\tfrac{\delta}{2}]$ for each $j$, we have $\mu(I_j)\leq {\mu}(A_j)+\rho$, and the result follows.
\end{proof}

We now deduce Theorem \ref{thm:Tasym} from Proposition \ref{prop:Ttrioclosed}.

\begin{proof}[Proof of Theorem \ref{thm:Tasym}]
Let $A,B\subset \mb{T}$ satisfy the assumptions in the theorem, namely
\begin{equation}\label{eq:asymassumps}
{\mu}(A+B) = {\mu}(A)+{\mu}(B)+ \rho < \tfrac{1}{2}\big(1+{\mu}(A)+{\mu}(B)\big),\quad  \rho < {\mu}(B) \leq {\mu}(A),\quad \rho<c.
\end{equation}
Note that this implies that $1-{\mu}(A+B) >\rho$. One could want to deduce from this that ${\mu}\big((A+B)^c\big)>\rho$ and then apply Theorem \ref{thm:Ttrio} to $A,B,-(A+B)^c$, but this raises several technical difficulties (in particular the behaviour of the inner Haar measure ${\mu}$ relative to taking complements). It is convenient to prove the result first for closed sets.

Suppose, then, that $A,B$ are closed and satisfy \eqref{eq:asymassumps}.  Fix any $\delta>0$ satisfying
\begin{equation}\label{eq:initas}
\mu(A)+\mu(B) + \rho + \delta < \tfrac{1}{2}\big(1+\mu(A)+\mu(B)\big), \quad \delta<\rho/4, \quad \mu(B)>\rho+\delta, \quad \rho+\delta < c.
\end{equation}
Note that the equality in \eqref{eq:asymassumps} is equivalent to $\mu(A)+\mu(B)+\mu\big(-(A+B)^c\big)=1-\rho$, which together with the first inequality in \eqref{eq:initas} implies that $\mu\big(-(A+B)^c\big)> \rho+2\delta$. In particular this, the last equality, and $\mu(B)>\rho+\delta$ together imply that
\begin{equation}\label{eq:initobs}
\mu(A) + 3(\rho+\delta) < 1\textrm{ ; \quad we have similarly \quad }\mu\big(-(A+B)^c\big) + 3 \rho+2\delta<1.
\end{equation}
Now let $A_1=A$, $A_2=B$, and let $A_3'$ be a closed subset of $-\big(A_1+A_2)^c$ satisfying $\mu(A_3')>\mu\big(-(A_1+A_2)^c\big)-\delta > \rho + \delta$. Let $X$ be the finite set of integers $n$ such that $\mu(n\cdot A_2)\leq \mu(A_2)+\rho +\delta$ (finite by Lemma \ref{lem:finset}, since $\mu(A_2)+\rho +\delta<1$ by \eqref{eq:initobs}). Let $A_3$ be the closed set given by applying Lemma \ref{lem:clapprox} to $A_3'\subset -(A_1+A_2)^c$ with $\varepsilon=\delta/2$.

We now show that $A_1,A_2,A_3$ satisfy the conditions to apply Proposition \ref{prop:Ttrioclosed} with initial parameter $\rho+\delta$. Firstly, as seen above, by construction we have $\mu(A_j)>\rho+\delta$ for $j=1,2,3$. Secondly, we have
\begin{equation}\label{eq:trioas1}
\mu(A_1)+\mu(A_2)+\mu(A_3)> \mu(A)+\mu(B)+ \mu\big(- (A+B)^c\big)- \delta = 1-(\rho+\delta).
\end{equation}
Finally, since $A_1+A_2+A_3$ is a closed set included in $A+B- (A+B)^c$, and since the latter set does not contain 0, the closed set $A_1+A_2+A_3$ must miss an entire open interval about 0, whence $\mu(A_1+A_2+A_3)<1$.

We can now apply Proposition \ref{prop:Ttrioclosed} to $A_1,A_2,A_3$ and thus obtain a positive integer $n$ and closed intervals $I_j'$ such that $\mu(I_j')\leq \mu(A_j)+\rho+\delta$ and $n\cdot A_j\subset I_j'$ for $j=1,2,3$. In particular $n$ must be in $X$, so by construction $n\cdot (A+B)^c$ is included in the closed interval $- I_3'+[-\tfrac{\delta}{2},\tfrac{\delta}{2}]$, and therefore so is $(n\cdot(A+B))^c$. Let $I_{3,\delta}=- I_3'+[-\tfrac{\delta}{2},\tfrac{\delta}{2}]$, thus  $\mu(I_{3,\delta})\leq \mu(A_3)+\rho+2\delta$, and let $I_{j,\delta}:=I_j'$ for $j=1,2$.

Now, we repeat the above argument for each term of a decreasing sequence of positive numbers $\delta_m$ satisfying \eqref{eq:initas} and tending to 0 as $m\to\infty$. Note that although the integer $n=n(\delta_m)$ could vary as $m$ increases, we can assume that it is constant, by passing to a subsequence if necessary, since $X$ is finite. For $j=1,2,3$ let $I_{j,m}=\bigcap_{k\leq m} I_{j,\delta_k}$. We have that $I_{j,m}$ is a closed interval for all $m$. Indeed, a priori the intersection of two intervals $I,J$ in $\mb{T}$ could be a union of two disjoint intervals, but this can occur only if $I\cup J =\mb{T}$, whereas here for $k<\ell$ we have by inclusion-exclusion that $\mu(I_{j,\delta_k}\cup I_{j,\delta_{\ell}})\leq \mu(A_j)+2\rho+4\delta_k$, which is less than 1 by \eqref{eq:initobs}. Therefore $(I_{j,m})_m$ is a decreasing sequence of closed intervals, each including $n\cdot A_j$ for $j=1,2$ and $(n\cdot(A+B))^c$ for $j=3$, whence the closed intervals $I_j=\bigcap_m I_{j,m}$ also include these sets respectively (in particular $n\cdot(A+B)$ includes the open interval $I_3^c$), and we have $\mu(I_j)\leq \mu(A_j)+\rho$.

This completes the proof of Theorem \ref{thm:Tasym} for closed sets $A$ and $B$.

\noindent Now let $A,B$ be arbitrary subsets of $\mb{T}$ satisfying \eqref{eq:asymassumps}. Let $\delta>0$ satisfy
\[
\rho+2\delta<c,\quad \delta<\frac13({\mu}(B)-\rho),\quad {\mu}(A)+{\mu}(B)+\rho+2\delta<\tfrac{1}{2}(1+{\mu}(A)+{\mu}(B)-2\delta).
\]
Then, let $A_1',A_2'$ be closed subsets of $A,B$ respectively with $\mu(A_1')>{\mu}(A)-\delta$ and $\mu(A_2')> {\mu}(B)-\delta$, and let $A_1,A_2$ be the closed sets obtained by applying Lemma \ref{lem:clapprox} with $A_1'\subset A$, $A_2'\subset B$, with $X$ being the finite set of integers $n$ such that $\mu(n\cdot A_1')\leq {\mu}(A)+\delta$. There is then $\rho'\leq \rho+2\delta$ such that
\[
\mu(A_1+A_2) = \mu(A_1)+\mu(A_2)+ \rho' < \tfrac{1}{2}\big(1+\mu(A_1)+\mu(A_2)\big),\quad  \rho' < \min(\mu(A_1),\mu(A_2)).
\]
Applying the result for closed sets we obtain closed intervals $I,J$, an open interval $K$, and a positive integer $n$ such that $n\cdot A\subset I+[-\tfrac{\delta}{2},\tfrac{\delta}{2}]$ and this closed interval has measure at most ${\mu}(A)+\rho+3\delta$, similarly $n\cdot B\subset J+[-\tfrac{\delta}{2},\tfrac{\delta}{2}]$ with $\mu(J+[-\tfrac{\delta}{2},\tfrac{\delta}{2}])\leq {\mu}(B)+\rho+3\delta$, and finally $n\cdot(A+B)\supset n\cdot(A_1+A_2)\supset K$ with $\mu(K)\geq \mu(A_1)+\mu(A_2)\geq {\mu}(A)+{\mu}(B)-2\delta$. Now letting $\delta \to 0$ in an argument similar to the one above for closed sets (taking a countable union of open intervals in the case of $K$), the result follows.
\end{proof}
\noindent In the case $A=B$, using Theorem \ref{thm:Zpsym} rather than Theorem \ref{thm:Zptrio} yields a better bound $c$.

\begin{proof}[Proof of Theorem \ref{thm:Tsym}]
Following a similar strategy as for Theorem \ref{thm:Ttrio}, we can reduce to the case of $A$ being a union of finitely many open intervals.
We replace the condition $\rho<c$ from Theorem \ref{thm:Tasym} by $\rho<\varepsilon\mu(A)$ with $\varepsilon<10^{-4}$ and use an argument similar to the proof of Proposition \ref{prop:SimpSet} to obtain a positive integer $n$ and a closed interval $I$ such that $n\cdot A \subset I$. Now Theorem \ref{thm:Zpsym}  does not give information on the structure of $A+A$, so we need to proceed differently to find some interval $K$ included in $n\cdot (A+A)$. 
Write $\tilde{A}=n\cdot A\subset\mb{T}$. We have $\mu(\tilde{A})\geq\mu(A)$ and, since $\tilde{A}\subset I$, we have
\[
\mu(\tilde{A}+\tilde{A})\leq 2\mu(I)\leq 2(\mu(A+A)-\mu(A))\leq (2+2\varepsilon)\mu(A)<3\mu(A)\leq 3\mu(\tilde{A}).
\]
We know that $\tilde{A}$ is included in an interval of length at most $\mu(A+A)-\mu(A)< 1/2$. The desired conclusion, i.e. that $\tilde{A}+\tilde{A}$ contains a large interval, is not affected by translating $\tilde{A}$ in $\mb{T}$, so we may suppose that $\tilde{A}\subset [0,1/2)$, where we identify $\mb{T}$ as a set with $[0,1)$. Then the sum $\tilde{A}+\tilde{A}$ behaves as a sum in $\mb{R}$, and so $\tilde{A}$ can be treated as a subset of $\mb{R}$ of doubling constant strictly less than $3$. Theorem 1 from \cite{dR} then  ensures the existence of an interval $K\subset \tilde{A}+\tilde{A}$ of length at least $2\mu(\tilde{A})\geq 2\mu(A)$, which completes the proof.
\end{proof}

\begin{remark}
Given the above deductions of Theorems \ref{thm:Ttrio}, \ref{thm:Tasym} and \ref{thm:Tsym} from their counterparts in $\Zmod{p}$, any improvement of the bounds $c$ in these discrete counterparts will immediately yield  the same improvement in the continuous setting.
\end{remark} 

\noindent In \cite{S-Z} Serra and Z\'emor give an example in $\Zmod{p}$ to show that the condition $|A+A|<\tfrac{p-3}{2} + |A|$ in Theorem \ref{thm:Zpsym} is necessary. We can adapt this to show that the condition ${\mu}(A+A)< \tfrac{1}{2}+{\mu}(A)$ is also necessary for Theorem \ref{thm:Tsym} to hold, as follows.
\begin{example}\label{ex:S-Z}
Viewing $\mb{T}$ as $[0,1]$ with addition mod 1, consider the set
\[
A= \big(\tfrac{1}{4}-\delta,\;\tfrac{1}{2}\big] \cup \big(1-\delta,\; 1\big] \subset \mb{T}, \textrm{ for an arbitrary fixed }\delta\in(0,\tfrac{1}{8}).
\]
We have $\mu(A)=\tfrac{1}{4}+2\delta$, and $A+A=\big(\tfrac{1}{2}-2\delta,\;1\big] \cup \big(1-2\delta,\;1\big]\cup\big(\tfrac{1}{4}-2\delta,\;\tfrac{1}{2}\big] =\big(\tfrac{1}{4}-2\delta,\;1\big]$. Hence $\mu(A+A)=\tfrac{3}{4}+2\delta=\tfrac{1}{2}+\mu(A)<3\mu(A)$. Moreover $\mu(A+A)=2\mu(A)+2\left(\frac18-\delta\right)$ and $\tfrac{1}{8}-\delta$ can be made arbitrarily small. However, we cannot include $n\cdot A$ in a preimage of a closed interval $I$ of measure $\mu(A+A)-\mu(A)=\tfrac{1}{2}$, for any positive integer $n$. Indeed, this is clear for $n=1$, as $A$ is not contained in an interval of length $\tfrac{1}{2}=\mu(A+A)-\mu(A)$. For $n\geq 2$, note that $\mu(n\cdot A)\geq \mu(n\cdot (\tfrac{1}{4}-\delta,\tfrac{1}{2}])$, and this is at least $\mu(2\cdot (\tfrac{1}{4}-\delta,\tfrac{1}{2}])$ (in general, for any interval $J\subset \mb{T}$ and any integers $n\geq m >0$, we have $\mu(nJ)\geq \mu(mJ)$). Since $\mu\big(2\cdot (\tfrac{1}{4}-\delta,\tfrac{1}{2}]\big)=\mu\big((\tfrac{1}{2}-2\delta,1]\big) > \tfrac{1}{2}$, we must have $D_n(A)> \tfrac{1}{2}$.
\end{example}

\begin{remark}\label{rem:BiluConj}
The conjecture of Bilu mentioned in the introduction, namely \cite[Conjecture 1.2]{Bilu}, proposes (in its special case for $\mb{T}$) that if $A,B\subset \mb{T}$ with $\alpha={\mu}(A)\geq {\mu}(B)=\beta$ satisfy ${\mu}(A+B)<\min(\alpha+2\beta,1)$, then there  exist closed intervals $I,J\subset\mb{T}$ and $n\in \mb{N}$ such that $n\cdot A\subset I$, $n\cdot B\subset J$, and $\mu(I)\leq {\mu}(A+B)-{\mu}(B)$, $\mu(J)\leq {\mu}(A+B)-{\mu}(A)$. Bilu proved that this conjecture holds under the additional condition that $\alpha/\tau\leq \beta \leq \alpha \leq c(\tau)$, where $c$ is some positive constant depending on $\tau\geq 1$; see \cite[Theorem 1.4]{Bilu}. However, note that  the conjecture itself does not hold for arbitrary $\alpha,\beta$. Indeed,  Example \ref{ex:S-Z} shows that the conjecture can fail for sets greater than $1/4$. These counterexamples can be ruled out by adding a condition to the conjecture, for instance that ${\mu}(A+B)<\tfrac{1}{2}(1+{\mu}(A)+{\mu}(B))$. Thus, a plausible version of Bilu's conjecture on $\mb{T}$, without a fixed upper restriction on ${\mu}(A)$, could be that Theorem \ref{thm:Tasym} holds for every $\rho\in (0,1)$. In another direction, one may try to find the largest upper bound on ${\mu}(A)$ under which Bilu's conjecture holds (given Example \ref{ex:S-Z}, this bound must be at most $1/4$).
\end{remark}

\section{Application to $k$-sum-free sets in $\mb{T}$}\label{sec:kfs}
\noindent A subset of an abelian group is said to be \emph{$k$-sum-free} if it does not contain any triple $(x,y,z)$ solving the linear equation $x+y=kz$, where $k$ is a fixed positive integer.\footnote{The term \emph{$k$-sum-free set} is used for instance in \cite{B&al}. These sets should not be confused with sets free of solutions to the equation $a_1+\cdots+a_k=b$, which have also been called $k$-sum-free sets (see \cite{Luczak}).} In the case $k=1$ the corresponding sets are known simply as sum-free sets, and their study dates back to work of Schur from 1916 \cite{Schur}. The case $k=2$ concerns sets avoiding 3-term arithmetic progressions, and this topic includes Roth's theorem from 1953 \cite{Roth} as well as the numerous related later works, recent examples of which include \cite{Bloom, CLP, E&G, Sanders}.  Note that this case differs in nature from the other cases, in that this is the only value of $k$ for which the linear equation in question is \emph{translation invariant}, meaning that if $(x,y,z)$ is a solution then so is $(x+t,y+t,z+t)$ for every fixed element $t$ in the group. 

For $k\geq 3$ the topic goes back at least to work of Erd\H os, who conjectured in particular that for large $n$ the odd numbers in $[n]$ form the unique $3$-sum-free set of maximum size (see \cite{C&G}). Chung and Goldwasser proved this conjecture in \cite{C&G}, and made an analogous conjecture about the maximum size of $k$-sum-free subsets of $[n]$ for $k\geq 4$, which was proved by Baltz, Hegarty, Knape, Larsson and Schoen in \cite{B&al}. Chung and Goldwasser also initiated the study of $k$-sum-free sets in the continuous setting. In particular, in \cite{C&G-cts} they determined the structure and measure of maximal $k$-sum-free Lebesgue measurable subsets of the interval $(0,1]$ for $k\geq 4$. They then made a conjecture concerning the structure and measure of maximal $3$-sum-free sets in this setting. Significant  progress toward this conjecture was made by Matolcsi and Ruzsa in \cite{M&R}, and the conjecture was then fully proved by Plagne and the second named author in \cite{D&P}.

Here we initiate the study of $k$-sum-free sets in $\mb{T}$ by considering the problem of estimating the following quantity:
\[
d_k(\mb{T})=\sup\{\mu(A): A\textrm{ is a Haar measurable $k$-sum-free subset of }\mb{T}\}.
\]
Note that $A$ is $k$-sum-free if and only if $(A+A)\cap k\cdot A=\emptyset$, and since by Raikov's inequality we have $\mu(A+A)\geq 2\mu(A)$, it follows that 
\begin{equation}\label{eq:CDbound}
3\mu(A)\leq \mu(A+A)+\mu(k\cdot A)= \mu((A+A)\cup k\cdot A)\leq 1, \;\textrm{ so }\mu(A)\leq 1/3.
\end{equation}
Given this, the problem of determining $d_1(\mb{T})$ is easily settled: in $\mb{T}$ viewed as $[0,1)$ with addition mod 1, the interval $(\tfrac{1}{3},\tfrac{2}{3})$ is a sum-free set of maximum measure $1/3$.

For $k=2$, it follows from the above-mentioned invariance of the equation $x+y=2z$ that $d_2(\mb{T})=0$ (in fact any set $A\subset \mb{T}$ of measure $\alpha>0$ must contain a positive measure $c(\alpha)$ of 3-term progressions; see for instance  \cite[Theorem 1.4]{CS2}). 

Let us now focus on $k\geq 3$. Here we can improve on \eqref{eq:CDbound} as follows.
\begin{theorem}\label{thm:ksfs-ub}
Fix any $\varepsilon>0$ for which Theorem \ref{thm:Tsym} holds, and let $k\geq 3$ be an integer. Then $d_k(\mb{T}) \leq \max\{\frac{1}{3+\varepsilon}, \frac{1+k\varepsilon}{k+2}\}$. 
\end{theorem}
\noindent The greatest value of $\varepsilon$ currently available here is the one provided by Serra and Z\'emor in \cite{S-Z}, namely $\varepsilon=10^{-4}$. This gives us $d_k(\mb{T}) \leq \tfrac{1}{3+10^{-4}}$ for all $k\geq 3$.

We prove Theorem \ref{thm:ksfs-ub} in several steps.

\noindent For a set $X\subset \mb{T}$ and $n\in \mb{N}$, we denote by $n^{-1}X$ the set $\{t\in \mb{T}: n\,t \in X\}$. Note that $X\subset \mb{T}$ is $k$-sum-free if and only if $X\cap\, k^{-1}(X+X)=\emptyset$. The following lemma tells us that if $A$ is $k$-sum-free and has measure close to $1/3$ then for some $n\in \mb{N}$ we must have $n\cdot A$ contained efficiently in an interval $I$ that is \emph{almost} $k$-sum-free, in the sense that $I\cap\, k^{-1}(I+I)$ has small measure.
\begin{lemma}\label{lem:alt}
Let $k\geq 3$ be an integer, let $A\subset \mb{T}$ be a $k$-sum-free Borel set, and let $\varepsilon\leq 10^{-4}$. Then either $\mu(A)\leq 1/(3+\varepsilon)$ or there exists a closed interval $I\subset \mb{T}$ and a positive integer $n$ such that $A\subset n^{-1}I$, $\mu(I)\leq \mu(A)(1+\varepsilon)$, and $\mu\big(I\cap\,k^{-1}(I+I)\big)\leq 2 \varepsilon\,\mu(I)$.
\end{lemma}

\begin{proof}
If $\mu(A+A)\geq (2+\varepsilon)\mu(A)$, then arguing as in \eqref{eq:CDbound} we deduce that $\mu(A)\leq 1/(3+\varepsilon)$. We may therefore assume that $\mu(A+A)\leq (2+\varepsilon)\mu(A)$. Applying Theorem \ref{thm:Tsym} with $\varepsilon$, we obtain an interval $I$ with $\mu(I)\leq \mu(A+A)-\mu(A)$ and $n\in \mb{N}$ such that $A\subset n^{-1}I$.

Letting $B=n^{-1}I$ and using that the map $x\mapsto nx$ is measure-preserving, we have $\mu(B)=\mu(I)$, and so
\begin{equation}\label{eq:1st}
\mu(B\setminus A)=\mu(I)-\mu(A)\leq \varepsilon\, \mu(A) \leq \varepsilon\, \mu(I).
\end{equation}
Note also that, since for every set $X\subset \mb{T}$ we have $n^{-1}(X+X)=n^{-1}X+n^{-1}X$, we have $\mu(B+B)=\mu\big(n^{-1}(I+I)\big)=2\mu(I)$ and so $\mu(B+B)\leq \mu(A+A)-\mu(A)+(1+\varepsilon)\mu(A) \leq \mu(A+A)+\varepsilon\,\mu(A)$. Hence
\begin{equation}\label{eq:2nd}
\mu\big((B+B)\setminus(A+A)\big) \leq \varepsilon\, \mu(I).
\end{equation}
Writing  $(B+B) = (A+A)\sqcup \big((B+B)\setminus(A+A)\big)$, we have
\[
B \cap \, k^{-1}(B+B)\; \subset \; \big[B \cap\, k^{-1}(A+A)\big]\, \sqcup \, k^{-1}\big[(B+B)\setminus \, (A+A)\big].
\]
Writing $B= A \sqcup ( B \setminus  A)$, and using that $A$ is $k$-sum-free, we have $B \cap\, k^{-1}(A+A) \, \subset \, B\setminus A$. Hence
\begin{equation}\label{eq:3rd}
B \cap \, k^{-1}(B+B) \; \subset \; (B\setminus A)\, \cup \, k^{-1}\big[(B+B)\setminus(A+A)\big].
\end{equation}
Combining \eqref{eq:1st}, \eqref{eq:2nd}, \eqref{eq:3rd}, and the fact that $\mu\big(I \cap \, k^{-1}(I+I)\big)=\mu\big(B \cap \, k^{-1}(B+B)\big)$, the result follows.
\end{proof}
\noindent Given this lemma, our goal now is to obtain a useful upper bound on the measure of an almost-$k$-sum-free interval. 

First we observe that if an interval $I\subset \mb{T}$ is $k$-sum-free then $\mu(I)\leq 1/(k+2)$. Indeed, we must have $k\cdot I$ disjoint from $I+I$, which implies that $\mu(I+I)+\mu(k\cdot I)\leq 1$, which in turn implies (since then $\mu(I+I)$ and $\mu(k\cdot I)$ are both less than 1) that $\mu(I+I)=2\mu(I)$ and $\mu(k\cdot I)=k\mu(I)$, which implies our claim. Note that this upper bound $1/(k+2)$ is attained by the interval $I=[\tfrac{2}{k^2-4},\tfrac{k}{k^2-4})$, which is indeed $k$-sum-free (a simple calculation shows that $(I+I)^c=k\cdot I$).

We now show that if an interval is almost $k$-sum-free, then its measure cannot be much larger than $1/(k+2)$.
\begin{lemma}\label{lem:int-estim}
Let $k$ be a positive integer, let $\delta\in [0,1)$, and let $I$ be a closed interval in $\mb{T}$ such that $\mu\big(I \cap \,k^{-1}(I+I) \big)\leq \delta\mu(I)$. Then $\mu(I)\leq \frac{1+k \delta / 2}{k+2}$.
\end{lemma}
\begin{proof}
Since $\mu\big(I \cap \,k^{-1}(I+I)\big)<\mu(I)$, we must have $I+I\neq \mb{T}$. The sumset $I+I$ is then a closed interval of measure $2\mu(I)$, and $k^{-1}(I+I)$ is a union of $k$ copies of $I+I$, each copy shrunk by a factor of $1/k$, and the centers of the copies forming an arithmetic progression of difference $1/k$. The complement of $k^{-1}(I+I)$ consists of $k$ components, each being an open interval of measure $\tfrac{1- 2\mu(I)}{k}$. Let $j$ be the number of these components that have non-empty intersection with $I$. Then $I$ must cover $j-1$ of the intervals making up $k^{-1}(I+I)$, so we have $(j-1) \tfrac{2\mu(I)}{k} \leq \mu\big(I\cap\, k^{-1}(I+I) \big) \leq \delta \mu(I)$, whence $j\leq 1+ \tfrac{\delta k}{2}$. We therefore have
\begin{eqnarray*}
\mu(I) & = & \mu\big(I \setminus \,k^{-1}(I+I)\big) + \mu\big(I \cap \,k^{-1}(I+I)\big)\\ 
&\leq & j\,\tfrac{1-2\mu(I)}{k} + \delta \mu(I) \leq (1+ \tfrac{\delta k}{2}) \tfrac{1-2\mu(I)}{k} + \delta \mu(I),
\end{eqnarray*}
whence $\mu(I) (1-\delta) k \leq (1+ \tfrac{\delta k}{2})(1-2\mu(I))$. 
After rearranging, we find that this inequality is equivalent to $\mu(I)\leq \frac{1}{2+k} + \frac{\delta}{2+4/k}$, and the result follows.
\end{proof}

\begin{proof}[Proof of Theorem \ref{thm:ksfs-ub}]
Combining Lemma \ref{lem:alt} with Lemma \ref{lem:int-estim} we have that either $\mu(A)\leq \frac{1}{3+\varepsilon}$ or there exist an interval $I$ and $n\in \mb{N}$ such that $\mu(I)\leq \frac{1+k \varepsilon }{k+2}$ and $n\cdot A\subset I$, whence $\mu(A)\leq \mu(n\cdot A)\leq \mu(I)$, and the result follows.
\end{proof}

\begin{remark}\label{rem:TtoZp}
There is an equivalence between determining $d_k(\mb{T})$ and determining the quantity $d_k(\mb{Z}_p)=\max\{\frac{|A|}{p}: A\subset \mb{Z}_p\textrm{ is $k$-sum-free}\}$ asymptotically as the prime $p$ tends to infinity. More precisely, it follows from \cite[Theorem 1.3]{CS} that $\lim_{p\to\infty} d_k(\Zmod{p})$ exists and equals $d_k(\mb{T})$. Theorem \ref{thm:ksfs-ub} therefore implies that this limit is at most $\max\{\frac{1}{3+\varepsilon}, \frac{1+k\varepsilon}{k+2}\}$.
\end{remark}

\medskip

\section{Application to sets of doubling less than 4 in $\mb{R}$}\label{sec:dlt4}
\noindent We shall write $\lambda$ for the inner Lebesgue measure on $\mb{R}$. For a bounded set $A\subset \mb{R}$ we denote by $\textrm{diam}(A)$ the diameter $\sup(A)-\inf(A)$.

The main result of this section is the following theorem which, for a bounded set $A\subset \mb{R}$ having doubling-constant not much larger than 3, gives information on the structure of $A$ modulo $\textrm{diam}(A)$.
\begin{theorem}\label{thm:small_doubling_R}
Let $\varepsilon\in [0,1)$ be such that Theorem \ref{thm:Tsym} holds. Let $A$ be a closed subset of $[0,1]$ satisfying $\lambda(A+A)\leq (3+\varepsilon)\lambda(A)$, $\lambda(A)\in \big(0,\frac{1}{2(1+\varepsilon)}\big)$, and $\textrm{diam}(A)=1$. Then there exists a positive integer $n\leq \frac{1+\varepsilon}{1-\varepsilon}$ such that $n\cdot A \bmod 1$ is included in a closed interval $I\subset \mb{T}$ with $\mu(I)\leq (1+\varepsilon)\lambda(A)$.
\end{theorem}
\noindent Below, when we use the notation $\lambda$ together with a sumset, then addition is meant to be in $\mb{R}$; when we use instead the notation $\mu$, addition is meant to be in $\mb{T}$.

\begin{remark}
Any progress on the upper bound for $\varepsilon$ in Theorem \ref{thm:Tsym} would yield progress in Theorem \ref{thm:small_doubling_R}. In particular, by \cite[Lemma 2]{FMY} (see also \cite[Corollary 1.5]{Bilu}), we already know that for some small absolute constant $a_0$, if we add to Theorem \ref{thm:Tsym} the assumption that $A\subset \mb{T}$ has ${\mu}(A)\leq a_0$, then the theorem holds for every $\varepsilon\in [0,1)$. This implies that for any set $A\subset [0,1]$ satisfying $\lambda(A) \leq a_0$ and $\lambda(A+A)\leq (3+\varepsilon)\lambda(A) < 4\lambda(A)$, there is a positive integer $n\leq \frac{1+\varepsilon}{1-\varepsilon}$ such that $n\cdot A \bmod 1$ is included in an interval $I\subset\mb{T}$ with $\mu(I)\leq (1+\varepsilon)\lambda(A)$. Furthermore, if Bilu's conjecture \cite[Conjecture 1.2]{Bilu} holds in the symmetric case $A=B\subset\mb{T}$ with ${\mu}(A)\leq 1/4$, then Theorem \ref{thm:small_doubling_R} 
 holds with the condition $\lambda(A)\in \big(0,\frac{1}{2(1+\varepsilon)}\big)$ replaced by $\lambda(A)\in (0, 1/4)$, and any $\varepsilon<1$.  
\end{remark}

\begin{remark}
Theorem \ref{thm:small_doubling_R} can be generalized to any bounded set $A\subset \mb{R}$, provided we replace the assumption $\lambda(A)\in \big(0,\frac{1}{2(1+\varepsilon)}\big)$ with $\lambda(A)\in \big(0,\frac{\textrm{diam}(A)}{2(1+\varepsilon)}\big)$ and that the conclusion is stated modulo $\textrm{diam}(A)$ rather than modulo 1.
\end{remark}

\begin{remark}
If $\varepsilon<1/3$ (this is the case for the $\varepsilon$ for which we know  that Theorem \ref{thm:Tsym} holds), then $n=1$. This means that under the hypothesis of Theorem \ref{thm:small_doubling_R} with $\varepsilon<1/3$, the set $A$ is included in a union $I_1\cup I_2$ of two intervals, $I_1$ being of the form $[0,a]$ and $I_2$ of the form $[1-b,1]$, with $a+b\leq \lambda(A+A)-\lambda(A)$.
\end{remark}
\begin{proof}[Proof of Theorem \ref{thm:small_doubling_R}]
By our assumptions $A\subset [0,1]$ is closed with $0,1\in A$. If $\mu$ is the inner Haar measure on $\mb{T}$ and $\tilde{A}$ denotes $A\bmod 1$, we have 
\[
\lambda(A+A)=\mu(\tilde{A}+\tilde{A})+\mu(\Sigma_2),
\]
where $\Sigma_2=\{x\in[0,1) \,:\, x,x+1\in A+A\}$. Since $0,1\in A$, we have that $A\setminus\{1\}$ is a subset of $\Sigma_2$, whence $\mu(\Sigma_2)\geq \mu(\tilde{A})=\lambda(A)$.  
Therefore $\lambda(A+A)\leq (3+\varepsilon)\lambda(A)$ implies that $\mu(\tilde{A}+\tilde{A})\leq (2+\varepsilon)\mu(\tilde{A})<\frac{1}{2}+\mu(\tilde{A})$. Theorem \ref{thm:Tsym} applied to $\tilde{A}$ gives us a positive integer $n$ such that $n\cdot\tilde{A}$ is included in a closed interval $I\subset \mb{T}$ of length at most $(1+\varepsilon)\mu(\tilde{A})=(1+\varepsilon)\lambda(A)$. We thus have $\tilde{A}\subset n^{-1}I$, and $n^{-1}I$ viewed as a subset of $[0,1)$ is a disjoint union of intervals $\frac{i}{n}+J$ mod 1, $i=0,\ldots,n-1$, where $J$ is a closed interval in $\mb{T}$ with $n\cdot J=I$ and $J\cap [0,\frac{1}{n})\neq \emptyset$. (Note that $J$ viewed as a subset of $[0,1)$ could have a part in $[1-\frac{1}{n},1)$.) Therefore, there exist sets $A_0\subset [0,\frac{1}{n})$, $A_i\subset (-\frac{1}{n},\frac{1}{n})$ for $i\in [n-1]$, and $A_n\subset (-\frac{1}{n},0]$, such that
\[
A= \bigcup_{i=0}^{n}\left(\frac{i}{n}+A_i\right),\textrm{ and }\; \bigcup_{i=0}^{n} A_i \bmod 1 \,\subset J, \textrm{ so in particular } \lambda\Big(\bigcup_{i=0}^{n} A_i\Big)\leq (1+\varepsilon)\frac{\lambda(A)}{n}.
\]
It remains to find an upper bound for $n$. We write $\alpha=\lambda(A)$, and $\alpha_i=\lambda(A_i)$ for $0\leq i\leq n$. Since $(1+\varepsilon)\alpha<\frac{1}{2}$, we have that $A+A$ is a disjoint union of subsets of $\mb{R}$ of the form
\[
A+A= \bigcup_{i=0}^{2n}\left(\frac{i}{n}+S_i\right) \quad \mbox{with}\quad S_i=\bigcup_{k,l\,:\, k+l=i}(A_k+A_l).
\]
In particular $A_i+A_i\subset S_{2i}$ and $A_i+A_{i+1}\subset S_{2i+1}$. This yields
\begin{align*}
\lambda(A+A)&= \sum_{i=0}^{2n}\lambda\left(S_i\right) = \sum_{i=0}^{n}\lambda\left(S_{2i}\right)+\sum_{i=0}^{n-1}\lambda\left(S_{2i+1}\right)\\
&\geq \sum_{i=0}^{n}2\alpha_i+\sum_{i=0}^{n-1}(\alpha_i+\alpha_{i+1})=4\alpha-(\alpha_0+\alpha_n).
\end{align*}
Now $\alpha_0 +\alpha_n \leq (1+\varepsilon)\alpha/n$, since mod 1 the sets $A_0\setminus\{0\}$ and $A_n$ are disjoint and their union is included in $J$. Hence 
$(3+\varepsilon)\alpha\geq \lambda(A+A)\geq 4\alpha-(\alpha_0+\alpha_n) \geq \alpha\left(4-\tfrac{1+\varepsilon}{n}\right)$, and this implies that $n\leq \frac{1+\varepsilon}{1-\varepsilon}$.
\end{proof}
\noindent In \cite{EGM}, Eberhard, Green and Manners prove the following result (see \cite[Theorem 6.2]{EGM}).
\begin{proposition}\label{corEGM}
Let $A\subset[0,1]$ be an open set with $\lambda(A-A)\leq 4\lambda(A)-\delta$. Then for some constant $c>0$ depending on $\delta$ there is an interval $I$ of length $\lambda(I)\geq c$ such that $\lambda(A\cap I)\geq (\frac{1}{2}+\frac{\delta}{7})\lambda(I)$.
\end{proposition}
\noindent Here $\lambda(A-A)$ can be replaced with $\lambda(A+A)$ (see Remark $(ii)$ after Theorem 6.2 in \cite{EGM}). Theorem \ref{thm:small_doubling_R} above yields an effective version of this result when $\varepsilon$ is close to $\lambda(A)$.

\begin{corollary}\label{cor:EGM}
Let $A\subset[0,1]$ be a non-empty closed set with $\lambda(A+A)\leq 4\lambda(A)-\delta$ and $\lambda(A) < \frac{\textrm{diam}(A)}{4}+\frac{\delta}{2}$, for some $\delta>0$. If $\delta>\lambda(A)(1-\varepsilon)$, with $\varepsilon$ such that Theorem \ref{thm:Tsym} holds, then there is an interval $I$ with $\lambda(I)\geq \min(\delta/4,\delta^2)$ such that $\lambda(A\cap I)\geq (\frac{1}{2}+\frac{\delta}{4})\lambda(I)$.
\end{corollary}

\begin{remark}
The size of $\delta$ is conditioned by Theorem \ref{thm:Tsym}. If Bilu's conjecture holds for sets $A$ with ${\mu}(A)\leq \tfrac{1}{4}$, then Corollary \ref{cor:EGM} gives an effective version of \cite[Theorem 6.2]{EGM} for sets $A$ with $\lambda(A)\leq \frac{\textrm{diam}(A)}{4}$.
\end{remark}

\begin{proof}
We first prove the result assuming that $\textrm{diam}(A)=1$. We use the notation introduced in the proof of Theorem \ref{thm:small_doubling_R}. 
Thus $A\subset[0,1]$ is a closed set with $0,1\in A$, with $\lambda(A+A)\leq 4\alpha-\delta$ and $\delta>\alpha(1-\varepsilon)$, where $\alpha=\lambda(A)$.

First suppose that $\delta> \alpha$. Then $\lambda(A+A)\leq 4\lambda(A)-\delta < 3\lambda(A)$. Note that since $\lambda(A+A)\geq 2\alpha$, we have $\delta\leq 2\alpha$. Applying Theorem \ref{thm:small_doubling_R} with $\varepsilon=0$, we obtain an interval $I\subset \mb{T}$ covering $A$ and with $\lambda(I)=\alpha$. Now as a subset of $\mb{R}$, the set $I$ is either an interval, in which case the conclusion holds, since $\lambda(I)\geq \alpha\geq \delta/2$ and $\lambda(A\cap I)=\lambda(I)$; or $I$ is a union of two intervals, one of which has measure at least $\lambda(I)/2\geq \delta/4$, and then this interval satisfies the desired conclusion.

Let us now suppose that $\delta\leq \alpha$, and write $\delta=(1-\varepsilon)\alpha$.\\
Then we have $\lambda(A+A)\leq (3+\varepsilon)\alpha$, and our assumption $\alpha<\frac{1}{4}+\frac{\delta}{2}$ also implies that $\alpha<\frac{1}{2(1+\epsilon)}$. Therefore, by Theorem \ref{thm:small_doubling_R}  there exists a positive integer $n \leq \frac{1+\varepsilon}{1-\varepsilon}$ such that 
\[
A= \bigcup_{i=0}^{n}\left(\frac{i}{n}+A_i\right) \quad \mbox{with}\quad \textrm{diam}_{\mb{T}}\left(\bigcup_{i=0}^n A_i\right)\leq (1+\varepsilon)\frac{\alpha}{n},
\]
where for a set $B\subset \mb{T}$ we denote by $\textrm{diam}_{\mb{T}}(B)$ the infimum of the Haar measures of intervals in $\mb{T}$ that cover $B$. Writing $\tilde{A_i}=\begin{cases}A_i&\mbox{if }i\in[n-1]\\A_0\cup A_m&\mbox {if }  i=0\end{cases}$, there exists $i\in[0,n-1]$ such that $\mu(\tilde{A_i})\geq \frac{\alpha}{n}$ and $\textrm{diam}_{\mb{T}}(\tilde{A_i})\leq (1+\varepsilon)\frac{\alpha}{n}$. There are now two cases.

If $i\not=0$, then letting $I$ be the interval $[\inf(A_i),\sup(A_i)]$ in $\mb{R}$, we have 
\[
\frac{\lambda(A\cap I)}{\lambda(I)}=\frac{\lambda(A_i)}{{\rm{diam}}(A_i)}\geq \frac{1}{1+\varepsilon}=\frac{1}{2-\delta/\alpha}\geq \frac12+\frac{\delta}{4\alpha}\geq \frac{1}{2}+\frac{\delta}{2},
\]
where for the last inequality we used that $\alpha<\frac{1}{2(1+\varepsilon)}\leq \frac{1}{2}$. We also have
\[
\lambda(I)\geq \frac\alpha{n}\geq \alpha \frac{1-\varepsilon}{1+\varepsilon}=\frac{\delta}{2-\delta/\alpha}\geq\frac{\delta}{2}.
\]
If $i=0$, so $\lambda(A_0)+\lambda(A_n)\geq \frac{\alpha}{n}$, then let $d_0={\rm{diam}}(A_0)$, $d_n={\rm{diam}}(A_n)$ and $\alpha_0=\lambda(A_0)$, $\alpha_n=\lambda(A_n)$. If $\tfrac{\alpha_i}{d_i}\geq \tfrac{1}{2}+\frac{\delta}{2}$ for both $i=0$ and $i=n$, then we choose $I=[0,d_0]$ if $d_0\geq d_n$, and $I=[-d_n,0]$ if $d_0\leq d_n$. We then have $\lambda(I)\geq \frac{\alpha}{2n}\geq \frac{\delta}{4}$, and $\lambda(A\cap I)=\tfrac{\alpha_i}{d_i} \lambda(I) \geq (\tfrac{1}{2}+\frac{\delta}{2})\lambda(I)$, so the desired conclusion holds. Otherwise, suppose that $\tfrac{\alpha_n}{d_n}< \tfrac{1}{2}+\frac{\delta}{2}$. Then
\[
\alpha_0 \geq \frac{\alpha}{n}-\alpha_n\geq \frac{\alpha}{n}-\left(\frac{1}{2}+\frac{\delta}2\right)d_n \geq \frac{\alpha}{n}-\left(\frac{1}{2}+\frac{\delta}{2}\right)\left((1+\varepsilon)\frac\alpha{n}-d_0\right).
\]
Using that $\varepsilon=1-\delta/\alpha$, the last term above is seen to equal
\begin{eqnarray*}
\frac{\alpha}{n}-\left(\frac{1}{2}+\frac{\delta}{2}\right)\left(\frac{2\alpha-\delta}{n}-d_0\right) & \geq & \left(\frac12+\frac{\delta}2\right)d_0+\frac{\delta}{2n}\left(1-2\alpha+\delta\right)\\
& > &\left(\frac12+\frac{\delta}2\right)d_0+\frac{\delta}{4n},
\end{eqnarray*}
where the last inequality used that $\alpha < \frac{1}{4}+\frac{\delta}{2}$.  
This implies on one hand that $\tfrac{\alpha_0}{d_0}\geq \tfrac12+\tfrac{\delta}{2}$, and on the other hand that $\alpha_0\geq \left(\frac12+\frac{\delta}{2}\right)\alpha_0+\frac{\delta}{4n}$, 
thus $\alpha_0\left(\frac12-\frac{\delta}{2}\right)\geq\frac{\delta}{4n}$, and so $\alpha_0 \geq \frac{\delta}{2n}\geq \frac{\delta}{2}\frac{1-\varepsilon}{1+\varepsilon}\geq\frac{\delta}{2}\frac{\delta}{2\alpha-\delta}\geq\delta^2$. Choosing $I=[0,d_0]$ yields the result. The case $\tfrac{\alpha_0}{d_0}< \tfrac12+\tfrac{\delta}{2}$ is similar. This completes the proof in the case $\textrm{diam}(A)=1$.

Finally, if $\textrm{diam}(A)< 1$, we may rescale the set in $\mb{R}$ defining $B= \frac{1}{\textrm{diam}(A)} A$ (and translate if necessary so that we may assume that $0,1\in B$). Applying the previous case to $B$, with parameter $\delta/\textrm{diam}(A)$, we obtain an interval $I_B$ satisfying $\lambda(I_B)\geq \min\big( \delta/\textrm{diam}(A), (\delta/\textrm{diam}(A))^2\big)$ and $\lambda(B\cap I_B) \geq (\frac{1}{2}+\frac{\delta}{4 \,\textrm{diam}(A)}) \lambda(I_B)$. The interval $I=\textrm{diam}(A)\cdot I_B$ satisfies the desired conclusion.
\end{proof}

\end{document}